\documentclass[12pt]{amsart}
\usepackage[english]{babel}
\usepackage{amsmath,amssymb,enumerate,amsthm}
\usepackage[T1]{fontenc}
\usepackage{cite}
\usepackage{latexsym}
\usepackage{setspace}
\usepackage{bm}
\usepackage{qtree} 
\usepackage{a4}
\usepackage{mathtools}

\newtheorem{thm}{Theorem}
\newtheorem{conjecture}{Conjecture}
\newtheorem{lemma}{Lemma}[section]
\newtheorem{theorem}[lemma]{Theorem}

\newtheorem{coro}[lemma]{Corollary}

\newtheorem{remark}[lemma]{Remark}
\def\N{\overline{N}}

\def\R{\mathbb{R}}
\def\Z{\mathbb{Z}}
\def\Q{\mathbb{Q}}

\def\SL{\mathrm{SL}(2,\mathbb{Z})}

\def\={\;=\;}
\def\.={\;\dot{=}\;}
\def\Im{\mathrm{Im}}
\def\Re{\mathrm{Re}}
\def\l{\ell}
\def\H{\mathcal{H}}

\def\G{\Gamma}

\DeclarePairedDelimiter\floor{\lfloor}{\rfloor}

\begin{document}
 \title[Values of modular functions at real quadratics]{Values of modular functions at real quadratics and   conjectures of Kaneko}

\author{P. Bengoechea}
\address{ ETH, Mathematics Dept.
\\CH-8092, Z\"urich, Switzerland}
\email{paloma.bengoechea@math.ethz.ch}
\thanks{Bengoechea's research is supported by SNF grant 173976.}

\author{ \"O. Imamoglu}
\address{ ETH, Mathematics Dept.
\\CH-8092, Z\"urich, Switzerland}
\email{ ozlem@math.ethz.ch}
\thanks{Imamoglu's  research is supported by SNF grant 200021-185014.}
\maketitle

\begin{abstract} In 2008,  M. Kaneko made several interesting observations  about the values of the modular $j$ invariant at real quadratic irrationalities. The  values of modular functions at real quadratics are defined in terms of their cycle integrals along the associated geodesics. In this paper we  prove some of the conjectures of  M. Kaneko for a general modular function. 
\end{abstract}  

 \section{Introduction}
 
 Let $\G=\SL$, $\H$ be the upper complex half-plane  and $j(z)= \frac{1}{q}+ 744+ 196884q+\cdots$ be the classical Klein  modular invariant. The values of $j$ at imaginary quadratic irrationalities have a long and rich history going back to Kronecker and Weber and play an important role in the theory of complex muliplication. They have also seen considerable recent interest due to the beautiful results of Borcherds and Zagier, which relate their traces to the coefficients of half integral weight modular forms. 
 
  For a real quadratic irrationality $w\in \Q(\sqrt D)$, where $D$ is a positive (not necessarily fundamental) discriminant, 
  the ``value" $f(w)$ of a holomorphic modular function $f$ at $w$ has  been defined only recently in \cite{kaneko} and \cite{DIT} using their periods  
  along the closed geodesic associated to $w$. We define
  $$f(w):=\int_{C_w}f(z) ds,$$
where $C_w$ is the closed geodesic  associated to $w$  in $\Gamma\backslash \H$ with an orientation that will be defined  later in the preliminaries (in p.3), 
and  $ds$ is the hyperbolic arc length.
  
%
 Even though the   properties of these  values and potential arithmetic applications remain inaccessible, Kaneko \cite{kaneko} 
 studied the numerical values of $j(w)$ and made several remarkable observations. 
Among   his many observations, we note the following boundedness  conjecture. 
\begin{conjecture}[Kaneko] Let   $w\in\Q(\sqrt D)$ be  a real quadratic irrationality and 
$\varepsilon>1$ be the smallest unit with positive norm in $\Q(\sqrt D)$.  Then 
 $$\Re (j^{nor}(w)) \in [j^{nor}((1+\sqrt 5)/2), 744]\,\,\,\,\, \mbox{and}\,\,\,\,\,
\Im  (j^{nor}(w)) \in (-1,1),$$ where    $$j^{nor}(w):=\dfrac{1}{2\log \varepsilon} j(w).$$

\end{conjecture}

In this paper we look at the values $f(w)$    for any modular function $f$ which 
takes real values on the geodesic arc $\{e^{i\theta}\colon \pi/3\leq \theta\leq 2\pi/3\}.$    

We denote by  $(a_0,a_1,a_2,\ldots)$  the `--' continued fraction
$$
(a_0,a_1,a_2,\ldots)=a_0-\dfrac{1}{a_1-\dfrac{1}{a_2-\dfrac{1}{\ddots}}},
$$
with  $a_i\in\Z$ and $a_i \geq 2$ for $i\geq 1$.
As a special case of our results we prove

\begin{thm}\label{thm-intro2}Let $j(z)$ be the classical  modular invariant. Let $w$ be a real quadratic irrational and $(\overline{a_1,\ldots,a_n})$ be its period 
in the negative continued fraction expansion.
 Then  we have

\begin{enumerate}
\item $\Re(j^{nor}(w))\leq 744.$
\medskip

\item  If  all  the partial quotients $a_r$ in the period of $w$ satisfy $a_r\geq 3M$   with $M=e^{55}$  then
 $$ 
 \Re(j^{nor}(w))\geq j^{nor}((1+\sqrt{5})/2).
 $$

\end{enumerate}
 \end{thm}

    Theorem~\ref{thm-intro2} proves the upper bound conjectured by Kaneko for real quadratic irrationalities 
    while Theorem~\ref{thm-intro1} below,  our second main result,  
    shows that  in fact this bound is optimal. 
 
 \begin{thm}\label{thm-intro1}Let $j(z)$ be the classical  modular invariant. Let $w$ be a real quadratic irrationality and $(\overline{a_1,\ldots,a_n})$ be its period 
in the negative continued fraction expansion.
 Then 
 for  any positive integer $N>2$, the value $j^{nor}((\overline{N}))$ for the quadratic irrationality $w=(\overline{N})$ is real and
$$ \lim_{N\rightarrow\infty} j^{nor}((\overline{N}))=744.
$$
\end{thm}

\medskip

   In\cite{DIT}, the values $j(w)$ were studied  not individually but in `traces'. By analogy to the traces of the values of $j$ at CM points, 
  the authors in \cite{DIT} define $\mbox{Tr}_D j:=\sum j(w)$, where the sum is  over  classes of indefinite binary quadratic forms of 
  fundamental discriminant $D$. They  relate these traces to the coefficients of mock modular forms, generalizing the results of Borcherds and Zagier. 
It was  conjectured in \cite{DIT} and proved in \cite{DFI} and \cite{masri} that 
 \begin{equation}\label{tr} \frac{\mbox{Tr}_ D (j)}{\mbox{Tr}_D(1)}\rightarrow 720, \end{equation} 
 as fundamental discriminants $D\rightarrow\infty.$ 
Theorem~\ref{thm-intro1}, when combined with the limiting behavior of the traces in  equation (\ref{tr}),   
has the following corollary. \begin{coro}\label{coro-intro}

 There are infinitely many discriminants of the form $D=N^2-4$ with class number greater than one and only finitely many such discriminants with class number one. 
 
 \end{coro}
 
  Corollary~\ref{coro-intro},  was also observed in  the master thesis of  S. P\"apcke \cite{papcke}  where another proof of Theorem~\ref{thm-intro1} was also given.
 This corollary is in the spirit of the conjectures of Chowla and Yokoi. Chowla (see \cite{C-F})  conjectured that  for a positive integer $p$  
 the class number  $h(4p^2+1)$ is greater than one   if  $p>13$. Whereas 
 Yokoi's conjecture says that $h(p^2+4)>1 $  if $p>17$ (see \cite{yokoi}). These conjectures   were proved effectively by A. Bir\'o  in \cite{biro1}
 and \cite{biro2}, and some generalizations were proved by Bir\'o and Lapkova  in \cite{biro-lapkova}.     
 Even though our result, unlike the results  of Bir\'o and Bir\'o-Lapkova, is not effective,   as far as we know it 
 provides the first direct application of cycle integrals of modular functions to   a classical   arithmetic question.

    It is worth noting that Part (2) of  Theorem~\ref{thm-intro2} can be rephrased in terms of the diophantine properties of the quadratic irrationalities, 
    more precisely in terms of the Lagrange  spectrum.  For any irrational number $x$, let
  $\|x\|$ denote  the distance from   $x$
to a closest integer. Then recall that the Lagrange spectrum $L$ is defined as $$
L:=\left\{\nu(x)\right\}_{x\in\R}\subseteq\left[0,1/\sqrt{5}\right]\quad\mbox{with}\quad
\nu(x)=\liminf_{q\rightarrow\infty} q\|qx\|.
$$
 It is known that if $x$  has a continued fraction $(a_1,a_2,\ldots)$ then  $\nu(x) \leq\inf_{r\geq 1} a_r^{-1}$. 
 Hence  part (2) of   Theorem~\ref{thm-intro2}  proves the lower bound  conjectured by Kaneko for the quadratic irrationalities $w$ with constant of 
 approximation $\nu(w)\in [0, 1/3M]$. 
 
 For the values $j(w)$ with $\nu(w)\in\left(1/3,1/\sqrt{5}\right]$, Kaneko observed more specific phenomena that have been partially proved in
 \cite{BI}. Every such quadratic irrationality  is $\SL$-equivalent to a Markov quadratic, i.e.
 a quadratic of the form
 $$
\dfrac{-3c+2k+\sqrt{9c^2 - 4}}{2c},
 $$
 where $0\leq k< c$, $k\equiv ab^{-1}\pmod{c}$ and $(a,b,c)\in\mathbb{N}^3$ is a solution to the Diophantine equation
$$
a^2+b^2+c^2=3abc.
$$
 A very rich theory has been developed by Markov for this sort of quadratic irrationalities  that the authors exploit in \cite{BI}.

 After giving some preliminaries in  the next section, we will collect several technical  results in  Section \ref{technical} that will be  needed 
 in the sequel.   In Section~\ref{main}, we start by proving   Theorem~\ref{thm-intro1} for  a holomorphic modular function which takes 
 real values on the geodesic arc $\{e^{i\theta}: \pi/3\leq\theta\leq 2\pi/3\}$. The  results from Section \ref{technical}  are then used to prove 
 Theorem~ \ref{main-thm}, which compares the values of modular functions at different quadratic irrationalities by comparing their corresponding partial 
 quotients.   Theorem \ref{main-thm} is the main theorem  of this paper and  the results (1) and (2) in  Theorem \ref{thm-intro2}     follow as  its simple corollaries. 

{\it Acknowledgement.} We thank the referee for his/her many useful remarks which  greatly improved the exposition of the paper.

 \section{Preliminaries}

 Let $w$ be a real quadratic irrationality and $\tilde w$ be its conjugate. 
$w$ and $\tilde w$ are the roots of a quadratic  equation
$$
ax^2+bx+c=0\qquad (a,b,c\in\Z, \quad(a,b,c)=1)
$$ 
with discriminant $D=b^2-4ac>0$.
We
choose $a,b,c$ such that $w=\frac{-b+\sqrt{D}}{2a}$, $\tilde w=\frac{-b-\sqrt{D}}{2a}$. 
 The geodesic $S_w$ in $\H$ joining $w$ and $\tilde w$ is given by the equation
$$
a|z|^2+b\, \Re(z)+c=0\qquad (z\in\H).
$$
The stabilizer
$\Gamma_w$ of $w$   preserves the quadratic form $Q_w=[a,b,c]$, and hence $S_w$. 
Let  $A_w$ be the generator of the infinite cyclic group $\Gamma_w$ with 
$$A_w=\begin{pmatrix} \frac{1}{2}(t-bu) &-cu\\au &\frac{1}{2}(t+bu)\end{pmatrix},$$
where $(t,u)$ is the smallest positive solution to Pell's equation $t^2-Du^2=4$.  We denote by $\varepsilon=(t+u\sqrt D)/2$ the smallest unit with 
positive norm.

For any modular function $f$, since the group $\Gamma_w$   preserves the expression $f(z)Q_w(z,1)^{-1}dz,$   one can define the cycle integral of  
$f$ along $C_w=\Gamma_w\backslash S_w$, also viewed as the ``value" of $f$ at $w$, by the complex number
$$
f(w):=\int_{C_w} \dfrac{\sqrt{D}f(z)}{Q_w(z,1)}dz.
$$
The factor $\sqrt{D}$ is introduced here since   $ds=\dfrac{\sqrt{D}f(z)}{Q_w(z,1)}dz$ on the geodesic $C_w$.
The integral defining $f(w)$   is $\Gamma$-invariant and can in fact be taken along any path in $\H$  from $z_0$ to $A_w z_0$, 
where $z_0$ is any point  in $\H$.  
Note that this gives an orientation on  $S_w$ from $w$ to $\tilde w$,  which is  counterclockwise if $a>0$ and clockwise if $a<0$. 
We  normalize the number $f(w)$ by the length of the geodesic $C_w$ which is given by
$$
\int_{ C_w} \dfrac{\sqrt{D}}{Q_w(z,1)}dz=2\log \varepsilon
$$
  and we define the normalized value as 
$$
f^{nor}(w):=\dfrac{f(w)}{2\log \varepsilon}.
$$

For a real quadratic irrationality $w$ with purely periodic continued fraction, it is known that the continued fraction expansion is given by setting
$w_0=w$ and inductively $w_{i+1}=w_i-1$ if $w_i\geq 1$ and $w_{i+1}=-1/w_i$ if $0<w_i<1$. This algorithm is cyclic, which implies that 
 the hyperbolic element $A_w$   is a word in negative powers of $T$ and $S$, where  
 $T=\tiny{\begin{pmatrix}1 &1\\0 &1\end{pmatrix}}$, $S=\tiny{\begin{pmatrix}0 &1\\-1 &0\end{pmatrix}}$ 
 (for more detail we refer the reader to \cite{Zb} and \cite[Proposition 2.6]{katok}).
 It follows that the algorithm 
\begin{equation}\label{wi}
w_0=w-1,\qquad
 w_{i+1}=A_{i+1}(w_{i})\qquad (i\geq 0),
\end{equation}
where
$$
 A_{i+1}=\left\{\begin{array}{ll}
T^{-1} &\mbox{if $w_{i}\geq 1 $},\\
V^{-1}=\tiny{\begin{pmatrix} 1 &0\\-1 &1\end{pmatrix}} &\mbox{otherwise},
\end{array}\right.
$$
 is also cyclic, since $V^{-1}=T^{-1}ST^{-1}$, and hence comes back to $w_0=w-1$ after a finite number of steps.
 We denote by $\ell_w$ the length of the cycle, so that $A_{w-1}=A_{\ell_w}\cdots A_1$.
 The next lemma is crucial in our work and is inspired by Zagier's argument in the proof of Theorem 7 in \cite{KZ}.

\begin{lemma}\label{lemma2} Let $w$ be a real quadratic irrationality with a purely periodic continued fraction expansion  and  let
\begin{equation}\label{kernel}
K(z,w)=\sum_{i=0}^{\l_w-1} \dfrac{1}{z-w_{i}}-\dfrac{1}{z-\tilde w_{i}}.
\end{equation}
Then 
\begin{equation}\label{**}
f(w)= \int_{\rho}^{\rho^2} f(z) K(z,w) dz, 
\end{equation}
where $\rho=e^{\pi i/3}.$ Here $w_i$'s are defined in (\ref{wi}) and $\tilde{w_i}$ is the Galois conjugate of $w_i$.
\end{lemma}

\begin{proof}
 Let $A_0=\mathrm{Id}$ and $z_i=A_0^{-1}\ldots A_i^{-1}\rho^2$ for $i\geq 0$.  We have 
\begin{align*}
 f(w)= f(w-1)&= 
 -\sqrt{D}  \int_{z_{0}}^{A_{w-1}^{-1} z_{0}} \dfrac{f(z)}{Q_{w-1}(z,1)}\, dz\\
 &= -\sqrt{D} \sum_{i=0}^{\l_w-1} \int_{z_{i}}^{z_{i+1}} \dfrac{f(z)}{Q_{w-1}(z,1)}\, dz\\
&=-\sqrt{D} \sum_{i=0}^{\l_w-1} \int_{\rho^2}^{A_{i+1} \rho^2} \dfrac{f(z)}{(Q_{w-1}|A_0^{-1} \ldots A_i^{-1})(z,1)}\, dz,\\
&= -\sum_{i=0}^{\l_w-1}\int_{\rho^2}^{A_{i+1} \rho^2}f(z)\left(\dfrac{1}{z-w_{i}}-\dfrac{1}{z-\tilde w_{i}}\right) dz.
\end{align*}
 Since $V(\rho^2)=T(\rho^2)=\rho$, we have proved the lemma.
\end{proof}

Let $w$ be a quadratic irrationality with purely periodic continued fraction and $\overline{a_1,\ldots,a_n}$ be the period in its continued fraction expansion. 
Each $w_i$ from \eqref{wi} has a continued fraction expansion of the form
$$
w_{r,k}=(k,\overline{a_{r+1},\ldots,a_n,a_1,\ldots,a_r}),\\
$$
with $1\leq r\leq n$ and for each $r$, $1\leq k\leq a_{r}-1$. The Galois conjugate of $w_{r,k}$ is 
$$
\tilde w_{r,k}=-(a_{r}-k,\overline{a_{r-1},\ldots,a_1,a_n,\ldots,a_{r}}).
$$
Hence
$$
K(z,w)=\sum_{r=1}^n \sum_{k=1}^{a_r-1}\dfrac{1}{z-w_{r,k}} - \dfrac{1}{z-\tilde w_{r,k}} .
$$
 
If $r$ is fixed 
and there is no possible confusion, we will  drop the dependence on $r$ and write  $w_k:=w_{r,k}$ for simplicity.

 \section{Technical lemmas}\label{technical}

 Let $w=(\overline{a_1,\ldots,a_n})$ and $v=(\overline{b_1,\ldots,b_m})$ be two purely periodic quadratic irrationals 
with $n\geq m$. If $n>m$, then we define $b_{m+1}, b_{m+2},\ldots, b_n$ by cycling back to $b_1,b_2,etc$. 
For a fixed $r,$ let  
\begin{equation}\label{1}
S_{w,v}(z,r):=\sum_{k=1}^{a_r-1}\Big(\dfrac{1}{z-w_{r,k}} - \dfrac{1}{z-\tilde w_{r,k}}\Big) 
- \sum_{k=1}^{b_r-1}\Big(\dfrac{1}{z-v_{r,k}} - \dfrac{1}{z-\tilde v_{r,k}}\Big).
\end{equation}

\begin{theorem}\label{S}
Let $w=(\overline{a_1,\ldots,a_n})$ and $v=(\overline{b_1,\ldots,b_m})$ be two purely periodic quadratic irrationals. 
If 
$a_r\leq b_r$ for all $r=1,\ldots, n$, then  we have for $\theta\in[\frac{\pi}{3},\frac{2\pi}{3}]$
 \begin{equation}\label{eq1th1}
C_2(r)<\Re(S_{w,v}(e^{i\theta},r))<C_1(r)
 \end{equation}
 with 
 \begin{align*}
 C_1(r)&=12.6 + 
  \dfrac{7}{3} \Big(\log\Big(\dfrac{b_r-3+k_1}{a_r-3+k_1}\Big)-\dfrac{4}{a_r-3+k_1}-\dfrac{3}{(a_r-3+k_1)^2}\Big),\\
 C_2(r)&=-3.925
+\Big(1-\frac{2}{a_r}+\frac{1}{2a_r^2}\Big)
\Big(2\log\Big(\dfrac{b_r+1+k_0}{a_r+2+k_0}\Big)-\dfrac{5}{a_r+2+k_0}\Big),
 \end{align*}
 
 and 
 \begin{equation}\label{eq2th1}
 -13.02<\Im(S_{w,v}(e^{i\theta},r))<15.
\end{equation}

 \end{theorem}
 
\bigskip

\noindent To prove   Theorem  \ref{S}, we start with  the following two  simple lemmas.

\bigskip
 
\begin{lemma}\label{lm1}

For $x,y\in\R$, $\theta\in[\frac{\pi}{3},\frac{2\pi}{3}]$, let
$$(x-y) F_{x,y}(\theta)=\Re\Big(\dfrac{1}{e^{i\theta}-x}-\dfrac{1}{e^{i\theta}-y}\Big).
$$

As a function of $\theta\in[\frac{\pi}{3},\frac{2\pi}{3}]$, $F_{x,y}(\theta)$ satisfies the following properties:

\begin{enumerate}[(i)]

\item\label{i} $F_{x,y}(\theta)$ is decreasing for $x,y\in (2,\infty)$ and increasing for $x,y\in(-\infty,-2)$.

\item\label{ii} For $0\leq |x|\leq 1$ and $0\leq |y|\leq 1$, we have
$$-1.4<F_{x,y}(\theta)<0.2.$$

\item\label{iii} For $1\leq |x|\leq 2$ and $1\leq |y|\leq 2$, we have 
$$-0.5<F_{x,y}(\theta)<0.2.$$

\item\label{iv} For $x\geq 2$ and $y\leq -2$, 
$$
 \dfrac{\min(xy -\frac{|x+y|}{2}-\frac{1}{2},xy-1)}{((-\frac{1}{2}+x)^2+\frac{3}{4})((\frac{1}{2}+y)^2+\frac{3}{4})}
 <F_{x,y} (\theta)< 
 \dfrac{xy+\frac{|x+y|}{2}-\frac{1}{2}}{((\frac{1}{2}+x)^2+\frac{3}{4})((-\frac{1}{2}+y)^2+\frac{3}{4})}.
$$

\end{enumerate}

\end{lemma}

\begin{proof} 
A simple calculation gives
\begin{equation}\label{F1}
F_{x,y}(\theta)=\dfrac{\cos^2\theta-(x+y)\cos\theta+xy-\sin^2\theta}{((\cos\theta-x)^2+\sin^2\theta)((\cos\theta-y)^2+\sin^2\theta)}.
\end{equation}

The assertion (iv) is straightforward using \eqref{F1}. The other three  assertions can be easily verified  numerically. For example,  for $x>2$ and $y>2$   the maximum of the derivative $\frac{dF}{d\theta}$ is $-0.00504971$ where as its minimum is  $-0.19245$. This shows that   for $x>2$ and $y>2,$ $\frac{dF}{d\theta}<0$ and hence $F_{x,y}(\theta)$ is decreasing.

\end{proof}

Similarly we have:

\begin{lemma}\label{lm2} 
For $x,y\in\R$, $\theta\in[\frac{\pi}{3},\frac{2\pi}{3}]$, let
$$(x-y) G_{x,y}(\theta)=\Im\Big(\dfrac{1}{e^{i\theta}-x}-\dfrac{1}{e^{i\theta}-y}\Big).
$$

As a function of $\theta\in[\frac{\pi}{3},\frac{2\pi}{3}]$, $G_{x,y}(\theta)$ satisfies the following 
properties:
\begin{enumerate}[(i)]
\item\label{2i} $G_{x,y}(\theta)$ is decreasing for $x,y\in (1,\infty)$ and for $x,y\in(-\infty,-1)$.

\item\label{2ii} $G_{x,y}(\theta)$ is increasing for $x\in(1,\infty) $ and $y\in(-\infty,-1)$.

 \item\label{2iii} For $0\leq |x|\leq 1$ and $0\leq |y|\leq 1$, we have
$$ G_{0,0}(e^{\pi i/3})\leq G_{x,y}(\theta) \leq G_{0,0}(e^{2\pi i/3}),$$
with $G_{0,0}(e^{\pi i/3})=-\frac{\sqrt{3}}{2}$ and $G_{0,0}(e^{2\pi i/3})=\frac{\sqrt{3}}{2}$.
\end{enumerate}

\end{lemma}

\begin{proof}
 Once again a simple calculation gives 
 \begin{equation}\label{G1}
G_{x,y}(\theta)=\dfrac{\sin\theta(x+y-2\cos\theta)}{((\cos\theta-x)^2+\sin^2\theta)((\cos\theta-y)^2+\sin^2\theta)}
 \end{equation}
and the proof of the lemma follows in a similar way to Lemma \ref{lm1}.
\end{proof}

\subsection{Proof of Theorem \ref{S}}

 We start by grouping  some of the terms from $S_{w,v}$ into two sums $S_1$ and $S_2$ 
($S_1$ corresponding to  terms $w_k,v_k$ and $S_2$ to   conjugates $\tilde w_k, \tilde v_k$)  
whose  contribution, as we will see, will be   minor. Since $r$ is fixed in the whole proof, we drop the dependence on $r$ 
in the notation for all these sums.   Define
$$
S_1(z):=\sum_{k=1}^{a_r-1}\dfrac{1}{z-w_k} -\dfrac{1}{z-v_k},
$$
$$
S_2(z):=\sum_{k=1}^{a_r-1} \dfrac{1}{z-\tilde w_k} - \dfrac{1}{z-\tilde v_{b_r-a_r+k}}.
$$
We group the remaining terms from $S_{w,v}$ in the sum $S_3$, which will be the major contribution:
$$
S_3(z)= \sum_{k=a_r}^{b_r-1} \dfrac{1}{z-v_k}-\dfrac{1}{z- \tilde v_{b_r-k}}.
$$
Hence, 
\begin{equation}\label{eq0}
S_{w,v}(z,r) = S_1(z)-S_2(z)-S_3(z).
\end{equation}

\subsubsection{Proof of inequality \eqref{eq1th1}.} 
The real part of $S_1(e^{i\theta})$ satisfies
$$
\Re(S_1(e^{i\theta}))=(w_0-v_0) \sum_{k=1}^{a_r-1} F_{w_k,v_k}(\theta).
$$
For $3\leq k\leq a_r-1$, since $\floor{w_k}=\floor{v_k}=k-1>2,$  by Lemma \ref{lm1} \eqref{i},
\begin{align*}
F_{w_k,v_k}(\theta)< F_{w_k,v_k}(\pi/3)
< \dfrac{k^2-k+\frac{1}{2}}{(k^2-3k+3)^2}
\end{align*}
and
$$
F_{w_k,v_k}(\theta)> F_{w_k,v_k}(2\pi/3)
> \dfrac{k^2-k-\frac{1}{2}}{(k^2+k+1)^2}.
$$
Since $-1<w_0-v_0<0$ and using Lemma \ref{lm1} \eqref{ii}-\eqref{iii}, we have
\begin{equation}\label{rin1.1}
\Re(S_1(e^{i\theta}))<(w_0-v_0) \Big(-1.9+\sum_{k=3}^{a_r-1} \dfrac{k^2-k-\frac{1}{2}}{(k^2+k+1)^2}\Big) < 1.683
\end{equation}
and similarly
\begin{equation}\label{rin1.2}
\Re(S_1(e^{i\theta}))>-0.4- \sum_{k=3}^{a_r-1} \dfrac{k^2-k+\frac{1}{2}}{(k^2-3k+3)^2} > -1.663.
\end{equation}
The real part of $S_2(e^{i\theta})$ satisfies
$$
\Re(S_2(e^{i\theta}))=(\tilde w_{a_r}-\tilde v_{b_r})\sum_{k=1}^{a_r-1}F_{\tilde w_k,\tilde v_{b_r-a_r+k}}(\theta).
$$
For $1\leq k\leq a_r-3$, $\floor{\tilde w_k}=\floor{\tilde v_{b_r-a_r+k}}=-a_r+k<-2$ and hence by Lemma \ref{lm1} \eqref{i}, we have
$$ 
F_{\tilde w_k,\tilde v_{b_r-a_r+k}}(\theta) < F_{\tilde w_k,\tilde v_{b_r-a_r+k}}(2\pi/3)
<\dfrac{(a_r-k)^2-(a_r-k)+\frac{1}{2}}{((a_r-k)^2-3(a_r-k)+3)^2}
$$
and
$$
F_{\tilde w_k,\tilde v_{b_r-a_r+k}}(\theta) > F_{\tilde w_k,\tilde v_{b_r-a_r+k}}(\pi/3)
> \dfrac{(a_r-k)^2-(a_r-k)-\frac{1}{2}}{((a_r-k)^2+(a_r-k)+1)^2}.
$$
Since $\floor{\tilde w_{a_r}}=\floor{\tilde v_{b_r}} =0$ and $b_r\geq a_r$, $  0<\tilde w_{a_r}-\tilde v_{b_r}<1$ and using Lemma \ref{lm1} \eqref{ii}-\eqref{iii},
\begin{align}
\Re(S_2(e^{i\theta})) &< 
0.4+\sum_{k=1}^{a_r-3}\dfrac{(a_r-k)^2-(a_r-k)+\frac{1}{2}}{((a_r-k)^2-3(a_r-k)+3)^2}
\nonumber
\\&< 0.4+\sum_{k=3}^{a_r-1} \dfrac{k^2-k+\frac{1}{2}}{(k^2-3k+3)^2} < 1.663 \label{rin2.1}
\end{align}
and
\begin{equation}\label{rin2.2}
 \Re(S_2(e^{i\theta})) > (\tilde w_{a_r}-\tilde v_{b_r}) \Big(-1.9+\sum_{k=1}^{a_r-3} \dfrac{(a_r-k)^2
 -(a_r-k)-\frac{1}{2}}{((a_r-k)^2+(a_r-k)+1)^2}\Big) >  -1.9.
 \end{equation}
The real part of $S_3(e^{i\theta})$ satisfies
$$
\Re(S_3(e^{i\theta}))=  \sum_{k=a_r}^{b_r-1} (2k+v_0-\tilde v_{b_r}) F_{v_k,\tilde v_{b_r-k}}(\theta).
$$
Since $-\floor{\tilde v_{b_r-k}}=\floor{v_k}=k,$ for $k\geq 3$,    by Lemma \ref{lm1} \eqref{iv},
$$
\dfrac{-k^2-\frac{3}{2}}{(k^2-3k+3)^2}
< F_{v_k,\tilde v_{b_r-k}}(\theta) < 
 \dfrac{-k^2+2k-\frac{1}{2}}{(k^2+k+1)^2}.
$$
 Therefore, similarly to above,
 \begin{align}
\Re(S_3(e^{i\theta}))&<0.6+ \sum_{k=a_r+k_0}^{b_r-1+k_0} (2k-1)
\dfrac{-k^2+2k-\frac{1}{2}}{(k^2+k+1)^2}\nonumber
\\
&<0.6
 -\sum_{k=a_r+2+k_0}^{b_r+1+k_0} (2k-5) \dfrac{1-\frac{2}{a_r}+\frac{1}{a_r^2}}{k^2}\nonumber\\
 &< 0.6-\Big(1-\frac{2}{a_r}+\frac{1}{2a_r^2}\Big)
 \Big(2\log\Big(\dfrac{b_r+1+k_0}{a_r+2+k_0}\Big)-\dfrac{5}{a_r+2+k_0}\Big)\label{rin3.1}
\end{align}
with $k_0=1$  if $a_r=2$ or $k_0=0$ otherwise. 
We also have
\begin{align}
\Re(S_3(e^{i\theta}))&>-2+ \sum_{k=a_r+k_1}^{b_r-1+k_1}
\dfrac{2k(-k^2-\frac{3}{2})}{(k^2-3k+3)^2}\nonumber\\
&>-9
-\dfrac{7}{3} \sum_{k=a_r-3+k_1}^{b_r-3+k_1} \dfrac{k+3}{k^2}\nonumber\\
&>-9-\dfrac{7}{3}\Big(\log\Big(\dfrac{b_r-3+k_1}{a_r-3+k_1}\Big)-\dfrac{4}{a_r-3+k_1}-\dfrac{3}{(a_r-3+k_1)^2}\Big)
\label{rin3.2}
\end{align}
with 
$$
k_1=\left\{\begin{array}{ll}
           2  &\quad\mbox{if $a_r=2$}\\
           1  &\quad\mbox{if $a_r=3$}\\
           0  &\quad\mbox{if $a_r\geq 4$}.
            \end{array}\right.
            $$
           By \eqref{eq0}, \eqref{rin1.1}, \eqref{rin2.2} and \eqref{rin3.2} we have that
\begin{equation}\label{re-ub}
\Re(S_{w,v}(e^{i\theta}),r)< 
12.6 + 
  \dfrac{7}{3} \Big(\log\Big(\dfrac{b_r-3+k_1}{a_r-3+k_1}\Big)-\dfrac{4}{a_r-3+k_1}-\dfrac{3}{(a_r-3+k_1)^2}\Big).
          \end{equation}
          By \eqref{rin1.2}, \eqref{rin2.1} and \eqref{rin3.1}, 
\begin{equation}\label{re-lb}\Re(S_{w,v}(e^{i\theta}),r)>- 3.926
+\Big(1-\frac{2}{a_r}+\frac{1}{2a_r^2}\Big)
\Big(2\log\Big(\dfrac{b_r+1+k_0}{a_r+2+k_0}\Big)-\dfrac{5}{a_r+2+k_0}\Big)
\end{equation}
           
\subsubsection{Proof of the inequality \eqref{eq2th1}}
The imaginary part of $S_1(e^{i\theta})$ satisfies
$$
\Im(S_1(e^{i\theta}))=(w_0-v_0) \sum_{k=1}^{a_r-1} G_{w_k,v_k}(\theta).
$$
For $k\geq 2$, by Lemma \ref{lm2} \eqref{2i},
$$
G_{w_k,v_k}(\theta)<G_{w_k,v_k}(e^{\pi i/3})<\dfrac{\frac{\sqrt{3}}{2}(2k-1)}{(k^2-3k+3)^2}
$$
and 
$$
G_{w_k,v_k}(\theta)>G_{w_k,v_k}(e^{2\pi i/3})>\dfrac{\frac{\sqrt{3}}{2}(2k-1)}{(k^2+k+1)}.
$$
Since $-1<w_0-v_0<0$ and using Lemma \ref{lm2} \eqref{2iii}, we have
\begin{equation}\label{iin1.1}
-4.209 <\Im(S_1(e^{i\theta}))<0.9.
\end{equation}
The imaginary part of $S_2(e^{i\theta})$ satisfies
$$
\Im(S_2(e^{i\theta}))=(\tilde w_{a_r}-\tilde v_{b_r})\sum_{k=1}^{a_r-1}G_{\tilde w_k,\tilde v_{b_r-a_r+k}}(\theta).
$$
For $1\leq k\leq a_r-2$, by Lemma \ref{lm2} \eqref{2i},
$$ 
G_{\tilde w_k,\tilde v_{b_r-a_r+k}}(\theta) < G_{\tilde w_k,\tilde v_{b_r-a_r+k}}(\pi/3)
<\dfrac{\sqrt{3}(k-a_r+\frac{1}{2})}{((a_r-k)^2+a_r-k+1)^2}
$$
and
$$
G_{\tilde w_k,\tilde v_{b_r-a_r+k}}(\theta) > G_{\tilde w_k,\tilde v_{b_r-a_r+k}}(2\pi/3)
> \dfrac{\sqrt{3}(k-a_r+\frac{1}{2})}{((a_r-k)^2-3(a_r-k)+3)^2}.
$$
Since $0<\tilde w_{a_r}-\tilde v_{b_r}<1$ and using Lemma \ref{lm2} \eqref{iii}, we have
\begin{equation}\label{iin2.1}
-6.187 <\Im(S_2(e^{i\theta}))<0.9.
\end{equation}
The imaginary part of $S_3(e^{i\theta})$ satisfies
$$
\Im(S_3(e^{i\theta}))=  \sum_{k=a_r}^{b_r-1} (2k+v_0-\tilde v_{b_r}) G_{v_k,\tilde v_{b_r-k}}(\theta).
$$
For $k\geq 2$, by Lemma \ref{lm2} \eqref{2ii}, 
$$
G_{v_k,\tilde v_{b_r-k}}(\theta)<G_{v_k,\tilde v_{b_r-k}}(2\pi/3)<\dfrac{\sqrt{3}}{(k^2-k+1)(k^2-3k+3)}
$$
 and 
 $$
 G_{v_k,\tilde v_{b_r-k}}(\theta)>G_{v_k,\tilde v_{b_r-k}}(\pi/3)>-\dfrac{\sqrt{3}}{(k^2-3k+3)(k^2-k+1)}.
 $$
 Therefore, similarly to previous cases we get
 \begin{equation} \label{iin3.1}
-7.905 < \Im(S_3(e^{i\theta}))< 7.905.
\end{equation}
 By \eqref{eq0}, \eqref{iin1.1}, \eqref{iin2.1} and \eqref{iin3.1}, we have
 \begin{equation}\label{im-ub} 
 -13.015< \Im(S_{w,v}(e^{i\theta},r))< 15.
\end{equation}

  \medskip
  
  The following corollary of Theorem \ref{S} is crucial for the next section.
  
  \begin{coro}  \label{cor} Let $w=(\overline{a_1,\ldots,a_n})$ and $v=(\overline{b_1,\ldots,b_m})$ be two purely periodic quadratic irrationals. 
  If $b_r\geq Ma_r$ for every $r$, then there exist   constants $K_1(M)$ and  $K_2(M)$ such that 
 $$K_2(M)<\cos\theta\Im( S_{w,v}(e^{i\theta},r))
+\sin\theta \Re( S_{w,v}(e^{i\theta},r))<K_1(M).$$
Moreover if $M\geq e^{55}$, then $K_1(M)$ and $ K_2(M)$ are positive.
  \end{coro}
  \begin{proof}This  follows easily from the bounds (\ref{re-ub}), (\ref{re-lb}), (\ref{im-ub}) together  with the simple observation that for $\pi/3\leq\theta\leq 2\pi/3$,  $-1/2\leq \cos\theta\leq 1/2$ and $\sqrt 3/2\leq\sin\theta \leq 1$.
  \end{proof}

 \section{Values of Modular functions}\label{main} 
 
 We start this section by 
 looking at the sequence of values $f^{nor}((\overline{N}))$, $N>2$ for a modular function which is real on the geodesic arc $\{e^{i\theta}\colon \pi/3\leq \theta\leq 2\pi/3\}.$ 
 We have 
\begin{theorem}\label{lim}
  Let $f(z)=\sum_{n\geq 0} c(n) q^n$ be a modular function which is 
 real valued on the geodesic arc $\{e^{i\theta}\colon \pi/3\leq \theta\leq 2\pi/3\}$. For any positive integer $N>2$,   
 the value $f^{nor}((\overline{N}))$ for the quadratic irrational $w=(\overline{N})$  is real and
$$
 \lim_{N\rightarrow\infty} f^{nor}((\overline{N}))=-\int_{\pi/3}^{2\pi/3}   f(e^{i\theta}) \sin\theta \,d\theta=c(0).
$$
In particular we have 
$$ 
 \lim_{N\rightarrow\infty} j^{nor}((\overline{N}))=744.
$$
 \end{theorem}
For $w=(\overline{N})$, the function $K_N(z):=K(z, w)$ defined by \eqref{kernel} satisfies the following  simple lemma.
\begin{lemma}\label{lema0}
 $$K_N(e^{i(\pi-\theta)})=\overline{K_N(e^{i\theta})}.
 $$
\end{lemma}

\begin{proof}
A simple calculation shows
\begin{equation}\label{re}
  \Re\Big(\dfrac{1}{e^{i\theta}-w_k}\Big)=\dfrac{2w_k(\cos^2\theta-w_k^2-\sin^2\theta)}{((\cos\theta-w_k)^2+\sin^2\theta)
  ((\cos\theta+w_k)^2+\sin^2\theta)}
  \end{equation}
 and
 \begin{equation}\label{im}
 \Im\Big(\dfrac{1}{e^{i\theta}-w_k}\Big)=\dfrac{-4w_k\sin\theta \cos\theta}{((\cos\theta-w_k)^2+\sin^2\theta)
  ((\cos\theta+w_k)^2+\sin^2\theta)}.
  \end{equation}
  These two equalities imply the lemma.
\end{proof}

 \begin{proof} 
 \noindent {\it  (Theorem \ref{lim})}:

Let $w=(\N)$. We have that $w=\frac{N+\sqrt{N^2-4}}{2}$ and $Q_w=[1,-N,1]$. Hence Pell's equation becomes
$$
t^2-(N^2-4)u^2=4,
$$
with $(N,1)$ being the smallest positive solution. Thus $\varepsilon_w=\frac{N+\sqrt{N^2-4}}{2}$.
Since $w_k=(k,\N)=-\tilde w_{N-k}$,  
 $$
 f^{nor}((\N))=\int_\rho^{\rho^2} \dfrac{f(z)}{2\log\varepsilon_w}  K_N(z) dz
 $$
 with  
 $$
K_N(z)=\sum_{k=1}^{N-1} \dfrac{1}{z-w_k}-\dfrac{1}{z+w_k}.
$$
 Now 
$$
\Re(f^{nor}((\N)))= -\int_{\pi/3}^{2\pi/3}\dfrac{f(e^{i\theta})}{2\log\varepsilon_w} \Big(\cos\theta
 \Im(K_N(e^{i\theta}))
+\sin\theta \Re(K_N(e^{i\theta})) \Big) d\theta
$$
and
$$
\Im(f^{nor}((\N)))= \int_{\pi/3}^{2\pi/3}\dfrac{f(e^{i\theta})}{2\log\varepsilon_w} \Big(\cos\theta
 \Re(K_N(e^{i\theta}))
-\sin\theta \Im(K_N(e^{i\theta}))\Big) d\theta.
$$
Writing $\int_{\pi/3}^{2\pi/3} = \int_{\pi/3}^{\pi/2} + \int_{\pi/2}^{2\pi/3}$
and using Lemma \ref{lema0}, we have that 
$$
\Im(f^{nor}((\N)))=0.
$$
Since $w_k=k-\frac{1}{a_{r+1}-\frac{1}{\cdots}},$ it follows from   $\varepsilon_w=\frac{N+\sqrt{N^2-4}}{2},$     \eqref{re} and \eqref{im}  that,  
for all $\theta\in[\frac{\pi}{3},\frac{2\pi}{3}]$,
 $$
 \lim_{N\rightarrow\infty} \dfrac{\cos\theta
 \Im(K_N(e^{i\theta}))
+\sin\theta \Re(K_N(e^{i\theta}))}{2\log \varepsilon_w} =\sin\theta.
 $$
 Thus
 $$
 \lim_{N\rightarrow\infty} f^{nor}((\N))=-\int_{\pi/3}^{2\pi/3}   f(e^{i\theta})\sin\theta \,d\theta= c(0).
 $$
    
 To see the  last equality  we  note that  $f$ is real on the arc $\{e^{i\theta}\colon \pi/3\leq \theta\leq 2\pi/3\}$ and hence the integral is just the parametrization of   $\int f(z) dz$ on this arc. On the other hand, due to the holomorphicity of  $f$  the integral is independent of the path of integration and  as such gives the  constant Fourier coefficient of $f$. 

 \end{proof}

Our next result  compares the  values of a modular function at two different quadratic irrationalities by comparing their corresponding partial quotients. The results  that were  given in  the introduction will then follow as corollaries of this main result. More precisely we have
 
 \begin{theorem}\label{main-thm} Let $f(z)=\sum_{n\geq 0} c(n) q^n$ be a modular function which is 
 real valued on the geodesic arc $\{e^{i\theta}\colon \pi/3\leq \theta\leq 2\pi/3\}$.  Suppose that 
\begin{equation}\label{condition}
\Re(f^{nor}((1+\sqrt{5})/2))<c(0).
\end{equation}
Then the following holds:  Let  $w$  and $v$ be two quadratic irrationals with respective periods $\overline{a_1,\ldots, a_n}$ 
and $\overline{b_1,\ldots,b_m}$ such that $m$ divides $n$. If $M=e^{55}$ and 
$b_r\geq M a_r$ for all $r=1,\ldots, n$,
then
  $$\Re(f^{nor}(w))<\Re(f^{nor}(v)).$$
\begin{remark} As the proof of   Theorem~\ref{main-thm} will show, the condition (\ref{condition}) can be replaced by the condition that $\Re(f^{nor}(\alpha))<c(0)$ for some quadratic irrational $\alpha$.
\end{remark}

%

\end{theorem}
\smallskip
\begin{proof}

We have that
 $$
f(w)-f(v)=\int_\rho^{\rho^2} f(z) \sum_{r=1}^n S_{w,v}(z,r) dz
$$
with $S_{w,v}(z,r)$ defined as in \eqref{1}, so
$$
\Re(f(w)-f(v))=-\int_{\pi/3}^{2\pi/3}f(e^{i\theta})\sum_{r=1}^n\cos\theta\Im( S_{w,v}(e^{i\theta},r))
+\sin\theta \Re( S_{w,v}(e^{i\theta},r))d\theta.
$$
By Theorem \ref{S}  and Corolary \ref{cor} we obtain
\begin{equation}\label{b1}
\Re(f(v))=\Re(f(w))+\mu(M)
\end{equation}
with
$$
nK_2(M)\int_{\pi/3}^{2\pi/3} f(e^{i\theta}) d\theta<\mu(M)
<nK_1(M)\int_{\pi/3}^{2\pi/3} f(e^{i\theta}) d\theta,
$$
$K_1(M)$, $K_2(M)$ being the positive constants from Corollary \ref{cor}.
In particular, if $f=1$, then
 \begin{equation}\label{log}
\log\varepsilon_v=\log\varepsilon_w+\lambda(M)
\end{equation}
with
$$
\frac{\pi}{3}nK_2(M) < \lambda(M) <\frac{\pi}{3}nK_1(M).
 $$
By definition, the inequality
\begin{align}\label{<}
\Re(f^{nor}(w))<\Re(f^{nor}(v))
\end{align}
holds if and only if
\begin{equation}\label{insandwich1}
 \Re(f(w)) \log\varepsilon_v < \Re(f(v)) \log\varepsilon_w. 
\end{equation}
 Now \eqref{b1} and \eqref{log} imply that 
\eqref{insandwich1} is equivalent to
  \begin{equation}\label{insandwich3}
   \Re(f^{nor}(v)) < \dfrac{\mu(M)}{\lambda(M)}
  \end{equation}
or also to
\begin{equation}\label{insandwich4}
   \Re(f^{nor}(w)) < \dfrac{\mu(M)}{\lambda(M)}.
  \end{equation}
  Since the last inequality does not depend on $v$ and is equivalent to \eqref{insandwich3}, the inequality 
    \eqref{insandwich3} holds either for all or for no  $v.$ 
     We first show that   $ \dfrac{\mu(M)}{\lambda(M)}> c(0). $
 Suppose on the contrary 
    $ \dfrac{\mu(M)}{\lambda(M)}\leq c(0). $  Let $v=(\overline{N}).$ For any $\epsilon>0$, for large enough $N>N_0$ using Theorem~\ref{lim} we   
    have $\Re (f^{nor}(v))>c(0)-\epsilon > \frac{\mu(M)}{\lambda(M)}$. 
 But then, since \eqref{insandwich3} is equivalent to \eqref{<}, for $N>\max{\{3M,N_0\}}$ and $w =\frac{1+\sqrt{5}}{2}$, we have  
  $$
  \Re(f^{nor}(w))= \Re(f^{nor}(1+\sqrt{5})/2)>\Re(f^{nor}(v)) > c(0)-\epsilon. $$ 
  Since this contradicts the assumption (\ref{condition}), we must indeed have $ \dfrac{\mu(M)}{\lambda(M)}> c(0). $
 Then with  $w= \frac{1+\sqrt{5}}{2}$,   we have $\Re(f^{nor}(w))=\Re(f^{nor}(\frac{1+\sqrt{5}}{2}))< c(0)<\dfrac{\mu(M)}{\lambda(M)}$. 
 Hence   for every $v=(\overline{b_1,\ldots,b_m})$ with $b_r>3M$ we have that
$$\Re(f^{nor}(v))<\dfrac{\mu(M)}{\lambda(M)}.$$
Now choose $  v_0=(\overline{b_1,\ldots,b_m})$ with $b_r>\max{\{3M,Ma_r\}}$. Then for this $v_0$ we have $\Re(f^{nor}(v_0))<\dfrac{\mu(M)}{\lambda(M)}$  
and hence  $\Re(f^{nor}(w))<\dfrac{\mu(M)}{\lambda(M)}.$  But this is equivalent to 
$$
\Re(f^{nor}(w))<\Re(f^{nor}(v)).
$$
 
 \end{proof}

We have the following immediate corollaries.

 \begin{coro}\label{coro1}
 For any quadratic irrational $w$,  and any modular function $f$ which satisfies the conditions of Theorem~\ref{main-thm} we have  
 $$\Re(f^{nor}(w))\leq -\int_{\pi/3}^{2\pi/3} f(e^{i\theta})\sin\theta\,d\theta=c(0).$$
 \end{coro}
  \begin{proof}For  any $w=(\overline{a_1,\ldots, a_n})$, choose   $N\in\mathbb N$ large enough, so that $N \geq e^{55}a_r$ for all $1\leq r\leq n$. Then by Theorem \ref{main-thm} we have 
  $$\Re( f^{nor}(w)) < \Re (f^{nor} (\overline{N})).$$
  Using  Theorem \ref{lim}, this gives 
   $$\Re( f^{nor}(w)) < \lim_{N\rightarrow\infty}\Re (f^{nor} (\overline{N}))=c(0).$$
 \end{proof}
 
\begin{coro}\label{coro2}
 Let $f$ be a modular function   which satisfies (\ref{condition}) and   $v$ be a quadratic irrational. If all the partial quotients in the period of $v$ are greater than 3M, then
 $$ 
 \Re(f^{nor}(v))\geq f^{nor}((1+\sqrt{5})/2).
 $$
\end{coro}

\begin{proof}
This follows from Theorem \ref{main-thm} applied to $w=\frac{1+\sqrt{5}}{2}=(2, \bar 3)$.
\end{proof}

Finally,  these corollaries prove Theorem~\ref{thm-intro2} since for the $j$ invariant the assumption  in the statement of  Theorem \ref{main-thm} is easily verified. Namely,  we have 
 
 $$
 \Re(j^{nor}((1+\sqrt{5})/2))=706.3248\ldots <744.$$


\begin{thebibliography}{9}




\bibitem{BI} P. Bengoechea, \"O. Imamoglu, \emph{Cycle integrals of the modular functions, Markov geodesics and a conjecture of Kaneko}, Algebra  Number Theory 13-4 (2019), 943-962.

\bibitem{biro1} Andr\'as Bir\'o, \emph{ Chowla's conjecture}, Acta.  Arith. 107 (2003) no. 2, 179--194.

 \bibitem{biro2} Andr\'as Bir\'o, \emph{Yokoi's conjecture}, Acta. Arith. 106 (2003) no.1, 85--104
 
 \bibitem{biro-lapkova} Andr\'as Bir\'o , Lapkova, Kastadinka, \emph{The class number one problem for the real quadratic fields $\Q(\sqrt{(an)^2+4a})$}, Acta Arith. 172 (2016) no2, 117--131
 
 \bibitem{C-F} S. Chowla and J. Friedlander, \emph{ Class numbers and quadratic residues}, Glasgow Math. J. 17 (1976), 47--52. 
 


\bibitem{DFI} W. Duke, J.B. Friedlander, H. Iwaniec, \emph{Weyl sums for quadratic roots}, Int. Math. Res. Not. IMRN 2012, no. 11, 2493--2549.

\bibitem{DIT} W. Duke, \"O. Imamoglu, A. Toth, \emph{Cycle integrals of the j-function and mock modular forms}, Ann. of Math. (2) 173 (2011), no. 2, 947--981.

\bibitem{kaneko}
M. Kaneko, \emph{Observations on the `values' of the elliptic modular function $j(\tau)$ at real quadratics}, Kyushu Journal of Mathematics 63 (2009), no. 2, 353-364.




\bibitem{katok} S. Katok  \emph{Coding of closed geodesics after Gauss and Morse.}  Geom. Dedicata 63 (1996), no. 2, 123--145. 

\bibitem{KZ}
W. Kohnen, D. Zagier, \emph{Modular forms with rational periods}, in: Modular Forms, R.A. Rankin (ed.), Ellis Horwood, Chichechester (1984) 197-249.

\bibitem{masri} R. Masri, \emph{The asymptotic distribution of traces of cycle integrals of the j-function}, 
Duke Math. J. 161 (2012), no. 10, 197--2000.

\bibitem{papcke} S. P\"apcke, \emph{Values of the j-function and its relation to Markoff numbers}, Master thesis, ETH Z\"urich. 

\bibitem{yokoi} H. Yokoi, \emph{ Class number one problem for certain kind of real quadratic fields}, in   Proc. Internat. Conf. (Katata, 1986), Nagoya Univ., Nagoya, 1986, 125-- 137.  

\bibitem{Zb}
D. Zagier: \emph{Zetafunktionen und quadratische K\"orper: eine Einf\"uhrung in die h\"ohere Zahlentheorie, Hochschultext}, Springer, Berlin, 1981.

\end{thebibliography}
\end{document}